\documentclass[12pt]{article}

\usepackage{lingmacros}
\usepackage{tree-dvips}
\usepackage{mathrsfs, amsfonts, amsmath, amsthm}

\newtheorem{theorem}{Theorem}
\newtheorem{lemma}{Lemma}
\newtheorem{remark}{Remark}
\newtheorem{definition}{Definition}

\author{Daniel Chun}
\title{Asymptotic Syzygies of Normal Crossing Varieties}
\date{}

\begin{document}

\maketitle

\begin{abstract}
Asymptotic syzygies of a normal crossing variety follow the same vanishing behavior as one of its smooth components, unless there is a cohomological obstruction arising from how the smooth components intersect each other. In that case, we compute the asymptotic syzygies in terms of the cohomology of the simplicial complex associated to the normal crossing variety.

We combine our results on normal crossing varieties with knowledge of degenerations of certain smooth projective varieties to obtain some results on asymptotic syzygies of those smooth projective varieties.
\end{abstract}

\subsection*{\centering Introduction}

Mark Green's paper [1] introduced a way to interpret syzygies of projective varieties as the cohomology groups of a Koszul type complex. This interpretation allowed for concrete computational results about syzygies that were not possible before. In particular, Lawrence Ein and Robert Lazarsfeld have established interesting results on vanishing and non-vanishing of asymptotic syzygies, as in [5] and [6]. Here, 'asymptotic' refers to the fact that they investigated the syzygies of smooth projective varieties of 'large enough' degree embedding.

Recently, Ziv Ran has extended some of these results to the case of nodal, possibly reducible, curves in [7]. In this paper, we try to generalize his results by analzying the case of normal crossing varieties of arbitrary dimension.

As one would expect intuitively, we find that the asymptotic syzygies of normal crossing varieties depend on the 'worst behaved' smooth component as well as the combinatorics of how the smooth components intersect each other. We use the knowledge of asymptotic syzygies of normal crossing varieties and degenerations in order to answer questions about asymptotic syzygies of smooth varieties in the Applications section.

Let's first go over some background material. Let $X$ be a projective variety of dimension $n$ defined over $\mathbb{C}$. Let $\mathscr{L}$ be a very ample line bundle on $X$, and let $\mathscr{B}$ be an arbitrary line bundle on $X$. $\mathscr{L}$ defines an embedding

\vspace{0.2in}

{\centering

$X \subseteq \mathbb{P}^r = \mathbb{P}H^0(\mathscr{L})$,

}

\vspace{0.2in}

where $r = h^0(\mathscr{O}_X(\mathscr{L}))$. The study of asymptotic syzygies is the study of syzygies of $X$ in $\mathbb{P}^r$ when $\mathscr{L}$ is very positive. Let $S = \textrm{Sym}H^0(\mathscr{L})$ be the homogeneous coordinate ring of $\mathbb{P}^r$, and let $\displaystyle R = \bigoplus_q H^0(\mathscr{B} \otimes q\mathscr{L}))$, and view $R$ as an $S$-module. $R$ has a minimal graded free resolution $E_{\bullet} = E_{\bullet}(X, \mathscr{B}, \mathscr{L})$,

\vspace{0.2in}

{\centering

$0 \rightarrow E_r = \bigoplus S(-a_{r,j}) \rightarrow ... \rightarrow E_1 = \bigoplus S(-a_{1,j}) \rightarrow E_0 = \bigoplus S(-a_{0,j}) \rightarrow R \rightarrow 0$

}

\vspace{0.2in}

Since Green's paper [1], much attention has been focused on what we can say about the syzygies in the above resolution, i.e. the degree $a_{p,j}$ of the generators of the $p^{th}$ module of syzygies of $R$. If $\mathscr{B} = \mathscr{O}_X$ and $\mathscr{L}$ is normally generated, then the above determines a resolution of the homogeneous ideal $I_X$ of $X$ in $\mathbb{P}^r$. In that case, the question is about the degrees of the equations defining $X$ and the syzygies among them. So, for example, if $a_{1, j}$ are all equal to 2, then $I_X$ is generated by quadrics, and if in addition, $a_{2, j}$ are all equal to 3, the syzygies among the quadratic generators are linear, and so on and so forth.

Green's main insight in [1] was to interpret these syzygies as the cohomology groups of a Koszul-type complex.

\vspace{0.2in}

\begin{definition}

Set $V = H^0(\mathscr{L})$. Let $K_{p,q}(X, \mathscr{B}, \mathscr{L})$ denote the cohomology of the Koszul-type complex

\vspace{0.2in}

{\centering

$\displaystyle ... \rightarrow \bigwedge^{p+1} V \otimes H^0((\mathscr{B} \otimes (q-1)\mathscr{L}_d)) \stackrel{\partial}{\rightarrow} \bigwedge^{p} V \otimes H^0((\mathscr{B} \otimes q\mathscr{L}_d)) \rightarrow \bigwedge^{p-1} V \otimes H^0((\mathscr{B} \otimes (q+1)\mathscr{L}_d)) \rightarrow ...$

}

\vspace{0.2in}

where the differential is given by $\displaystyle \partial (v_0 \wedge ... \wedge v_p \otimes m) = \sum_{k=0}^{p} (-1)^k v_0 \wedge ... \wedge \hat{v_k} \wedge ... \wedge v_p \otimes v_k m$.

\end{definition}

\vspace{0.2in}

Then, we get that $E_p (X, \mathscr{B}, \mathscr{L}) = \bigoplus K_{p,q}(X, \mathscr{B}, \mathscr{L}) \otimes_{\mathbb{C}} S(-p-q)$. In other words, $K_{p,q}(X, \mathscr{B}, \mathscr{L})$ is the $p^{th}$ syzygy of weight $q$ of the variety $X$ projectively embedded by $V$. If $\mathscr{B} = \mathscr{O}_X$, we omit $\mathscr{B}$ and write $K_{p,q}(X, \mathscr{L})$.

Using this interpretation, Green proved in [1] that for a smooth curve $C$ of genus $g$ and a line bundle $\mathscr{L}$ of degree $d$ on $C$,

\vspace{0.2in}

{\centering

$K_{p,q}(C, \mathscr{L}) = 0$ for $q \geq 3$ if $h^1(\mathscr{L}) = 0$, and

}

{\centering

$K_{p,2}(C, \mathscr{L}) = 0$ if $d \geq 2g + 1 + p$

}

\vspace{0.2in}

For higher dimensional varieties, the picture is more complicated, but there are still quite a few established results. For example, let $X$ be an abelian variety of dimension $n \geq 3$, $\mathscr{L}$ an ample line bundle on $X$, $a$ an integer with $a \geq 2$, and $\mathscr{B}$ a line bundle on $X$ such that $b\mathscr{L} - \mathscr{B}$ is ample for some integer $b \geq 1$. Set $r_a = h^0(a\mathscr{L}) - 1$, and assume $a \geq b$. M. Aprodu and L. Lombardi prove in [12] that

\vspace{0.2in}

{\centering

$K_{p,1}(X, \mathscr{B}, a\mathscr{L}) = 0$ for $p$ in the range $r_a - a(n-1) + b(1-\frac{1}{a}) \leq p \leq r_a - n$

}

\vspace{0.2in}

Our two main computational results in this paper have similar flavor to the above two established results. We prove that

\vspace{0.2in}

\textbf{Theorem 1} Let $X$ be a a very general smooth Calabi-Yau $n$-hypersurface in $\mathbb{P}^{n+1}$. Then, $K_{p, q}(X, \mathscr{O}_X (d)) = 0$ if $p \leq (q-1)d-3$

\vspace{0.2in}

and

\vspace{0.2in}

\textbf{Theorem 2} Let $X$ be a very general smooth degree $4a$ hypersurface in $\mathbb{P}^3$ with $a \geq 2$. Then, $X$ is a surface of general type with $K_{p, 1}(X, \mathscr{O}_X(d)) = 0$ if $p \geq h^0(d) - 4d + 4$ where $h^0(d) = h^0(\mathscr{O}_X(d))$

\vspace{0.2in}

\subsection*{\centering Notations}

Let $D = D_1 \cup ... \cup D_a$ be a normal crossing variety of dimension $n$ where the $D_i$ are the smooth irreducible components of $D$. Set $D_{i_0 ... i_p} = D_{i_0} \cap ... \cap D_{i_p}$, and assume all intersections between the components are transversal.

\vspace{0.2in}

Let $\mathscr{B}$ and $\mathscr{P}$ be arbitrary line bundles on $D$. Let $\mathscr{A}$ be an ample line bundle on $D$. Set $\mathscr{L}_d = \mathscr{P} \otimes d\mathscr{A}$ where $d >> 0$, and set $V = H^0(\mathscr{L}_d)$.

\vspace{0.2in}

Set $\displaystyle B^p_q = \bigoplus_{i_0 < ... < i_p} H^0((\mathscr{B} \otimes q\mathscr{L}_d)|_{D_{i_0 ... i_p}})$, and $\displaystyle B^p = \bigoplus_{q \geq 0} B^p_q$. Then $B^p$ is a graded $\displaystyle S(\bigoplus_{i_0 < ... < i_p} H^0(\mathscr{L}_d|_{D_{i_0 ... i_p}}))$-module. Letting $\displaystyle V \rightarrow \bigoplus_{i_0 < ... < i_p} H^0(\mathscr{L}_d|_{D_{i_0 ... i_p}})$ be the natural map induced by restriction maps to each component $D_{i_0 ... i_p}$, we see $B^p$ is also a graded $S(V)$-module by the action of $S(V)$ induced by this map.

\subsection*{\centering Koszul Cohomology of Normal Crossing Varieties}

By [4], we have the following complex of sheaves on $D$, which we will call (*).

\vspace{0.2in}

{\centering

$\displaystyle 0 \rightarrow \bigoplus_{i_0} \mathscr{O}_{D_{i_0}} \rightarrow \bigoplus_{i_0 < i_1} \mathscr{O}_{D_{i_0 i_1}} \rightarrow ... \rightarrow \bigoplus_{i_0 < ... < i_{a-2}} \mathscr{O}_{D_{i_0 ... i_{a-2}}} \rightarrow \mathscr{O}_{D_{1...a}} \rightarrow 0$

}

\vspace{0.2in}

The above complex is exact everywhere except at $\displaystyle \bigoplus_{i_0} \mathscr{O}_{D_{i_0}}$. And the cohomology at that step is $\mathscr{O}_D$.

\vspace{0.2in}

For $d >> 0$ and for any choice of $i_0 < ... < i_p$, we have by Serre vanishing,

\vspace{0.2in}

{\centering

$H^i((\mathscr{B} \otimes q\mathscr{L}_d)|_{D_{i_0 ... i_p}}) = 0 = H^i(\mathscr{B} \otimes q\mathscr{L}_d)$ for $i > 0$, $q >0$ ... (1)

}

\vspace{0.2in}

We also have

\vspace{0.2in}

{\centering 

$H^0((\mathscr{B} \otimes q\mathscr{L}_d)|_{D_{i_0...i_p}}) = 0 = H^0(\mathscr{B} \otimes q\mathscr{L}_d)$ if $q < 0$ ... (2)

}

\vspace{0.2in}

Serre vanishing also gives us $H^1(\mathscr{L}_d (-D_{i_0...i_p})) = 0$ where $\mathscr{L}_d (-D_{i_0...i_p})$ is the coherent sheaf of sections of $\mathscr{L}_d$ which vanish along $D_{i_0...i_p}$, which means

\vspace{0.2in}

{\centering

Restriction map $\phi_{i_0...i_p} : V \rightarrow H^0(\mathscr{L}_d|_{D_{i_0 ... i_p}})$ is surjective ... (3)

}

\vspace{0.2in}

Taking global sections of (*) tensored by $\mathscr{B} \otimes q\mathscr{L}_d$, we get a complex $B^{\bullet}_q$, with

\vspace{0.2in}

{\centering

$H^0(\mathscr{B} \otimes q\mathscr{L}_d) = H^0(B^{\bullet}_q)$ ... (4)

}

\vspace{0.2in}

For any $i > 0$, $q > 0$, we see by (1) that $0 \rightarrow \mathscr{B} \otimes q\mathscr{L}_d \rightarrow$ (*) is an acyclic resolution of $\mathscr{B} \otimes q\mathscr{L}_d$, thus

\vspace{0.2in}

{\centering

For any $i > 0$, $q > 0$, $H^i(B^{\bullet}_q) = H^i(\mathscr{B} \otimes q\mathscr{L}_d) = 0$ ... (5)

}

\vspace{0.2in}

Fix some $l \in \mathbb{N}$ (we will specify later on in this report what value we need $l$ to be). Set $\displaystyle C^{p, q} = \bigwedge^{l-q}V \otimes B^p_q$ to be the double complex with horizontal differentials coming from $(-1)^p$ times the maps for the complex $B^{\bullet}_q$ and vertical differentials coming from the Koszul complex maps.

\vspace{0.2in}

Then we get two spectral sequences $'E$ starting from horizontal differentials and $''E$ starting from vertical differentials, with same abutment. By (4) and (5),

\vspace{0.2in}

{\centering

$'E_2^{0, q} = K_{l-q, q}(D, \mathscr{B}, \mathscr{L}_d)$ and $\displaystyle 'E_2^{p, 0} = \bigwedge^l V \otimes H^p(B^{\bullet}_0)$ for $p >0$

with zeroes everywhere else on the $'E_2$ page ... (6)

}

\vspace{0.2in}

We also have

\vspace{0.2in}

{\centering

$''E_1^{p, q} = K_{l-q, q}(B^p, V)$ ... (7)

}

\vspace{0.2in}

Now, let's start calculating some of the terms in the $''E_1$ page. $K_{l-q, q}(B^p, V)$ is the cohomology at the middle of

\vspace{0.2in}

{\centering

$\displaystyle ... \rightarrow \bigwedge^{l-q+1} V \otimes \bigoplus_{i_0 < ... < i_p} H^0((\mathscr{B} \otimes (q-1)\mathscr{L}_d)|_{D_{i_0 ... i_p}}) \rightarrow \bigwedge^{l-q} V \otimes \bigoplus_{i_0 < ... < i_p} H^0((\mathscr{B} \otimes q\mathscr{L}_d)|_{D_{i_0 ... i_p}}) \rightarrow \bigwedge^{l-q-1} V \otimes \bigoplus_{i_0 < ... < i_p} H^0((\mathscr{B} \otimes (q+1)\mathscr{L}_d)|_{D_{i_0 ... i_p}}) \rightarrow ...$

}

\vspace{0.2in}

The above complex is a direct sum over all $i_0 < ... < i_p$ of complexes of the form

\vspace{0.2in}

{\centering

$\displaystyle ... \rightarrow \bigwedge^{l-q+1} V \otimes H^0((\mathscr{B} \otimes (q-1)\mathscr{L}_d)|_{D_{i_0 ... i_p}}) \rightarrow \bigwedge^{l-q} V \otimes H^0((\mathscr{B} \otimes q\mathscr{L}_d)|_{D_{i_0 ... i_p}}) \rightarrow \bigwedge^{l-q-1} V \otimes  H^0((\mathscr{B} \otimes (q+1)\mathscr{L}_d)|_{D_{i_0 ... i_p}}) \rightarrow ...$

}

\vspace{0.2in}

By (3), $\displaystyle \bigwedge^{l-q} V$ has a filtration with quotients $\displaystyle \bigwedge^j \textrm{ker}\phi_{i_0...i_p} \otimes \bigwedge^{l-q-j} H^0(\mathscr{L}_d|_{D_{i_0...i_p}})$, as $j = 0, ..., h^0(\mathscr{L}_d)-h^0(\mathscr{L}_d|_{D_{i_0...i_p}})$, which induces a filtration on the above complex with quotients each of which is a tensor product of a fixed vector space $\displaystyle \bigwedge^j \textrm{ker}\phi_{i_0...i_p}$ with

\vspace{0.2in}

{\centering

$\displaystyle ... \rightarrow \bigwedge^{l-q+1-j} H^0(\mathscr{L}_d|_{D_{i_0...i_p}}) \otimes H^0((\mathscr{B} \otimes (q-1)\mathscr{L}_d)|_{D_{i_0 ... i_p}}) \rightarrow \bigwedge^{l-q-j} H^0(\mathscr{L}_d|_{D_{i_0...i_p}}) \otimes H^0((\mathscr{B} \otimes q\mathscr{L}_d)|_{D_{i_0 ... i_p}}) \rightarrow \bigwedge^{l-q-1-j} H^0(\mathscr{L}_d|_{D_{i_0...i_p}}) \otimes H^0((\mathscr{B} \otimes (q+1)\mathscr{L}_d)|_{D_{i_0 ... i_p}}) \rightarrow ...$

}

\vspace{0.2in}

Note $K_{l-q-j, q}(D_{i_0...i_p}, \mathscr{B}|_{D_{i_0...i_p}}, \mathscr{L}_d|_{D_{i_0...i_p}})$ is the cohomology at the middle of the above Koszul complex.

\subsection*{\centering A Technical Lemma}

The main lemma is as follows.

\vspace{0.2in}

\begin{lemma}

Suppose for each choice of $q$, we're given a number $s_q$ such that

\vspace{0.2in}

{\centering

$K_{h^0(\mathscr{L}_d|_{D_{i_0...i_p}}) - s, q}(D_{i_0...i_p}, \mathscr{B}|_{D_{i_0...i_p}}, \mathscr{L}_d|_{D_{i_0...i_p}}) = 0$ for all $0 \leq s \leq s_q$ and for any choice of $i_0 < ...< i_p$.

}

\vspace{0.2in}

Then, for any $q$ and $l$ with $0 \leq q \leq n+1$ and $l-q \geq h^0(\mathscr{L}_d)-s_q$, we get $\displaystyle K_{l-q, q}(D, \mathscr{B}, \mathscr{L}_d) = 0$ if $q = 0$ or $1$ and $\displaystyle K_{l-q, q}(D, \mathscr{B}, \mathscr{L}_d) \cong \bigwedge^l V \otimes H^{q-1}(B^{\bullet}_0)$ if $2 \leq q \leq n+1$.

\end{lemma}

\begin{proof}
Fix any $l$ with $l-q \geq h^0(\mathscr{L}_d)-s_q$. We then have

\vspace{0.2in}

{\centering

$K_{l-q-j, q}(D_{i_0...i_p}, \mathscr{B}|_{D_{i_0...i_p}}, \mathscr{L}_d|_{D_{i_0...i_p}}) = 0$ for all $j \leq h^0(\mathscr{L}_d) - h^0(\mathscr{L}_d |_{D_{i_0...i_p}})$,

}

\vspace{0.2in}

and for $j > h^0(\mathscr{L}_d) - h^0(\mathscr{L}_d |_{D_{i_0...i_p}})$ we have $\displaystyle \bigwedge^j \textrm{ker}\phi_{i_0...i_p} = 0$, which means

\vspace{0.2in}

{\centering

$''E_1^{p, q} = K_{l-q, q}(B^p, V) = 0 =$ $''E_{\infty}^{p, q}$ ... (8)

}

\vspace{0.2in}

Recall that $'E$ and $''E$ have the same abutment. Then, by (6) and (8), we have the proof.
\end{proof}

\vspace{0.2in}

\begin{remark}
If we set $\mathscr{B} = \mathscr{O}_D$ and assume that $H^i(\mathscr{O}_{D_{i_0...i_p}}) = 0$ for all $i >0$ and $i_0 < ... < i_p$, then $B^{\bullet}_0$ gives us an acyclic resolution of $\mathscr{O}_D$.
\end{remark}

So, in this case, $H^i(B^{\bullet}_0) = H^i(\mathscr{O}_D)$. Furthermore, by \textbf{Remark 5.5} in [4], $H^i(\mathscr{O}_D)$ depends entirely on the simplicial complex $\Delta(D)$, which is constructed using only the incidence information among the $D_{i_0...i_p}$.

In conclusion, under these assupmtions, the behavior at the tail of a row in the Betti table of $(D, \mathscr{L}_d)$ depends only on the combinatorics of how the pieces $D_{i_0...i_p}$ intersect and on the behavior at the tail of a row in the Betti table of each $(D_{i_0...i_p}, \mathscr{L}_d|_{D_{i_0...i_p}})$.

\vspace{0.2in}

\begin{remark}
Recall $n = \textrm{dim}D$. Set $D_i =$ Proj $\mathbb{C}[x_0, ..., x_{n+1}]/x_{i-1}$ and $D = D_1 \cup ... \cup D_{n+2}$. Set $\mathscr{B} = \mathscr{O}_D$. Then, $\Delta(D) = S^n$, the $n$-sphere. We're in the case of \textbf{Remark 2}, so $H^i(B^{\bullet}_0) = H^i(\mathscr{O}_D) \cong H^i(S^n, \mathbb{C}) = \mathbb{C}$ if $i = 0$ or $n$ and is $0$ for all other values of $i$. Thus, by \textbf{Lemma 1}, $\displaystyle K_{l-q, q}(D, \mathscr{L}_d) = 0$ where $l \geq h^0(\mathscr{L}_d)+q-s_q$.
\end{remark}

In fact, we would be able to use the exact same argument for any normal crossing variety $D$ with $H^i(\Delta(D)) = 0$ for any $1 \leq i \leq n-1$ for any $n = \textrm{dim}D$.

\subsection*{\centering Applications}

In this section, we use the upper semicontinuity of dimension of Koszul Cohomology groups in flat families with constant cohomology to deduce vanishing statements for asymtotic syzygies of smooth projective varieties. Specifically, we obtain results on syzygies of smooth Calabi-Yau hypersurfaces of arbitrary dimension and smooth hypersurfaces of general type in $\mathbb{P}^3$.

\vspace{0.2in}

\textbf{Calculation 1}: Consider $F \subseteq \mathbb{P}^{n+1}$ $\times$ $\mathbb{P}^{1} =$ Proj $\mathbb{C}[x_0, ..., x_{n+1}]$ $\times$ Proj $\mathbb{C}[y_0, y_1]$, defined by $y_0 f + y_1 g = 0$, where $f$ is a homogeneous degree $n+2$ polynomial cutting out a smooth hypersurface in $\mathbb{P}^{n+1}$, and $g = x_0 x_1 ... x_{n+1}$.

Then, $F \rightarrow \mathbb{P}^1$ is a flat family where very general fibers $F_t$ for $t \neq 0$ are smooth Calabi-Yau $n$-folds, and the special fiber $F_0$ is a union of $n+2$ copies of $\mathbb{P}^n$ intersecting each other in a spherical configuration.

Let's prove vanishing statements on the special fiber $F_0$. By \textbf{Theorem 2.2} in [2], $K_{p, q}(\mathbb{P}^n, \mathscr{O}_{\mathbb{P}^n}(-3), \mathscr{O}_{\mathbb{P}^n}(d)) = 0$ if $(q-1)d-3 \geq p$. Set $h^0(d) = h^0(\mathscr{O}_{\mathbb{P}^n}(d))$. By duality of Koszul Cohomology groups, this means $K_{h^0(d)-n-1-p, n+1-q}(\mathbb{P}^n, \mathscr{O}_{\mathbb{P}^n}(d)) = 0$ if $(q-1)d-3 \geq p$.

Note that Euler characteristic is locally constant for the fibers of a flat family. In addition, since we're assuming $d >> 0$, by Serre Vanishing, the higher cohomologies of all the fibers vanish. Lastly, $\mathbb{P}^1$ is connected, thus the $h^0$ term is constant for all fibers of $F \rightarrow \mathbb{P}^1$. Set $h$ to be this constant term.

Then, by \textbf{Remark 2} and \textbf{Lemma 1}, we find for any fixed $q$ with $0 \leq n+1-q \leq n+1$ that $K_{h-n-1-p, n+1-q}(F_0, \mathscr{O}_{F_0}(d)) = 0$ if $(q-1)d-3 \geq p$. By upper semicontinuity, this means $K_{h-n-1-p, n+1-q}(F_t, \mathscr{O}_{F_t}(d)) = 0$ if $(q-1)d-3 \geq p$ for a very general fiber $F_t$. Since, the dualizing sheaf $K_{F_t}$ is $\mathscr{O}_{F_t}$, by duality, we obtain the following result.

\vspace{0.2in}

\begin{theorem}
Let $X$ be a a very general smooth Calabi-Yau $n$-hypersurface in $\mathbb{P}^{n+1}$. Then, $K_{p, q}(X, \mathscr{O}_X (d)) = 0$ if $p \leq (q-1)d-3$.
\end{theorem}

\vspace{0.2in}

\begin{remark}
Before this paper, the best result on vanishing of asymptotic syzygies of smooth Calabi-Yau varieties was \textbf{Corollary 1.6} in [8], which states that for a smooth Calabi-Yau $n$-fold $X$, $K_{p, q}(X, \mathscr{O}_X (d)) = 0$ for all $p$ and $q$ with $p \leq d-n$ and $q \geq 2$. So \textbf{Theorem 1} is an improvement on that result in the particular case of $X$ being a smooth Calabi-Yau $n$-hypersurface.
\end{remark}

\vspace{0.2in}

\textbf{Calculation 2}: First, note that a very general quartic $K3$ hypersurface in $\mathbb{P}^3$ has Picard number 1. Fix a positive integer $m$. Consider $F \subseteq \mathbb{P}^3$ $\times$ $\mathbb{P}^1$ = Proj $\mathbb{C}[x_0, x_1, x_2, x_3]$ $\times$ Proj $\mathbb{C}[y_0, y_1]$, defined by $y_0 f + y_1 g = 0$, where $f$ is a homogeneous degree $4a$ polynomial cutting out a smooth hypersurface in $\mathbb{P}^3$, and $g = g_1 g_2 ... g_a$ where $g_i$ are $a$ very general homogeneous polynomials of degree $4$ each cutting out a smooth hypersurface in $\mathbb{P}^3$ with Picard number 1.

Then, $F \rightarrow \mathbb{P}^1$ is a flat family where very general fibers $F_t$ for $t \neq 0$ are smooth surfaces of general type of genus $\binom{4a-1}{3}$, and the special fiber $F_0 = S_1 \cup ... \cup S_a$ where each $S_i$ is a smooth quartic $K3$ hypersurface in $\mathbb{P}^3$.

Let's prove vanishing statements on the special fiber $F_0$. We will use \textbf{Theorem 1.3} in $[9]$, which gives us a complete description of vanishing and non-vanishing of syzygies of K3 surfaces.

Let $\mathscr{L}$ be a line bundle on a $K3$ surface $S$ with $\mathscr{L} = 2g-2$ where $g$ is the genus of any member of $|\mathscr{L}|$. Note $h^0(\mathscr{L}) = g+1$. By [11], the Clifford index of any irreducible smooth curve $C \in |\mathscr{L}|$ is constant. Call this constant $c$. Then, the theorem tells us that $K_{p, 1}(S, \mathscr{L}) = 0$ if and only if $p \geq g-c-1 = h^0(\mathscr{L})-c-2$. Assume that Picard number of $S$ is 1.

By section 2 of [10], setting $H$ to be a generator of the Picard group of $S$, we get $c = H \cdot (C-H) - 2$. In our case, $S = S_i$ and $\mathscr{L} = \mathscr{O}_{S_i}(d)$, thus, $c+2 = 4d-4$, which means $K_{p, 1}(S_i, \mathscr{O}_{S_i}(d)) = 0$ if and only if $p \geq h^0(\mathscr{O}_{S_i}(d)) - 4d + 4$. Applying \textbf{Lemma 1} here, we obtain $K_{p, 1}(F_0, \mathscr{O}_{F_0}(d)) = 0$ if and only if $p \geq h - 4d + 4$, where as in \textbf{Calculation 1}, $h$ is defined to be the constant $h^0$ term of all the fibers of the flat family $F \rightarrow \mathbb{P}^1$. We can apply \textbf{Lemma 1} now to get

\begin{theorem}
Let $X$ be a very general smooth degree $4a$ hypersurface in $\mathbb{P}^3$ with $a \geq 2$. Then, $X$ is a surface of general type with $K_{p, 1}(X, \mathscr{O}_X(d)) = 0$ if $p \geq h^0(d) - 4d + 4$ where $h^0(d) = h^0(\mathscr{O}_X(d))$.
\end{theorem}

The above result complements the work of F. J. Gallego and B. P. Purnaprajna on the syzygies of surfaces of general type in [8].

\end{document}